\documentclass[11pt,reqno,a4paper]{amsart}
\usepackage[utf8]{inputenc}
\usepackage[english]{babel}
\usepackage{dsfont}
\usepackage{amsaddr}
\usepackage{amsmath}
\usepackage{amsthm}
\usepackage{amsfonts}
\usepackage{amssymb}
\usepackage{graphicx}
\usepackage{mathrsfs}
\usepackage{xcolor}
\usepackage{bm}

\def\Rset{\mathbb{R}}

\def\ds{\displaystyle}

\newtheorem{lem}{Lemma}
\newtheorem{definition}{Definition} 

\newtheorem{prop}{Proposition}
\newtheorem{assumption}{Assumption}

\everymath{\displaystyle}

\title[A new proof of the competitive exclusion principle]{A new proof of the competitive exclusion principle in the chemostat}
\date{\today}

\author{Alain Rapaport}
\address{MISTEA (Mathematics, Informatics and Statistics for
  Environmental and Agronomic Sciences), Univ. Montpellier, INRA,
  Montpellier SupAgro, Place Pierre Viala, 34060 Montpellier, France}
\email{alain.rapaport@inra.fr}

\author{Mario Veruete}
\address{IMAG (Institut Montpéllierain Alexander Grothendieck),
  Univ. Montpellier, CNRS, Place Eugène Bataillon, 34095, Montpellier, France}
\email{mario.veruete@umontpellier.fr}

\begin{document}

\begin{abstract}
We give an new proof of the well-known competitive
exclusion principle in the chemostat model with $n$ species competing
for a single resource, for any set of increasing growth functions. 
The proof is constructed by induction on the number of the species,
after being ordered. It uses elementary analysis and comparisons of
solutions of ordinary differential equations.\\


\end{abstract}

\maketitle

\section{Introduction}

The Competitive Exclusion Principle (CEP) has a long history in the
scientific literature.
Since the thirties, the Russian botanist Gause conducted experiments on the
growth of yeasts and paramecia in mixed cultures, and reported that the
most competitive species systematically eliminated the other
\cite{G32}. In his book ``The struggle for existence'', he showed that
competitive exclusion had indeed a more universal scope: two
similar living organisms evolving in the same environment and competing
for a shared resource cannot coexist for ever, one of them having
always a slight advantage over the other one, or being more adapted to
the ecosystem \cite{G34}.
In the 1960s, this statement has become quite popular in ecology but
also in economics: the CEP applies to many kinds of ecosystems, and
not only for microorganisms, since there are consumers and resources
\cite{H60}. It is also commonly taught as ``Gause's law'' in theoretical ecology.\\

However, it was not until the 1970s that the first statement of a mathematical
theorem was found in the literature, along with its proof
\cite{HHW77}, for the chemostat model.
It refers to the mathematical result that
establishes conditions under which almost all solutions converge
toward a steady-state of the system having at most one species.
The chemostat model is widely used in microbiology and ecology as a
mathematical representation of growth of micro-organisms in ecosystems that are
continuously fed with nutrients. Several textbooks on the mathematical
analysis of this model with one and more species are available \cite{SW95,AA11,HLRS17}.
The chemostat model can be also considered as a quite general resource/consumers model \cite{MBN03}. 
The CEP has also a long history in the literature of bio-mathematics. Hsu, Hubbell and Waltman have
proposed a first proof in 1977 for Monod's functions as particular
growth rates \cite{HHW77}. Hsu generalized this result in 1978 for different removal rates \cite{H78}.
These two contributions use explicit Lyapunov functions to demonstrate the
overall convergence. In
1980, Armstrong and McGehee have given a simple proof for any monotonic growth
functions but for particular
initial conditions belonging to an invariant set \cite{AM80}.
In 1985, Butler and Wolkowicz proposed a proof for any monotonic growth
function \cite{BW85}. One of the difficulties to prove the global
stability originates from the fact that the graphs of any growth
function can intersect one another at several points. During the
transients a species could dominate the competition without being the
final winner of the competition on the long run.  
Finally, it was in 1992 that Wolkowicz and Lu proposed a proof, based
on a Lyapunov function, for growth functions more general than Monod functions
(but under additional technical assumptions) and different removal
rates \cite{WL92}. 
This result has been later extended or complemented \cite{L98,RH08,RDH09,SM11,S13}.
However, the proof of global stability for any
monotonic growth functions and removal rates remains today an open
mathematical problem \cite{LLS03}.\\

In the present paper, we propose a new
proof of the CEP for any monotonic growth functions but under identical
removal rates. The existing proofs rely on relatively sophisticated
tools, such as $\omega$-limit sets \cite{BW85}, Lyapunov functions
and LaSalle Invariance Principle \cite{HHW77,WL92} and the theory
of asymptotically autonomous systems (e.g. Appendix F in \cite{SW95}).
We show here that it
is possible to obtain a proof with elementary analysis, based on
single comparisons of solutions of ordinary differential equations.
While species are sorted in ascending break even concentrations,
the key of the proof relies on the observation of the time evolution of
the proportions $r_{i}$ of the concentration of species $i$ over the
one of species $1$. 
Whatever is the initial condition and how the
transients could exhibit an alternation of
dominance among species or not, there always exists a finite time at the
end of which the proportion $r_{n}$ is decreasing exponentially for any
future time. We show that this property is due to the
level of the resource that reaches in finite time an interval which is
unfavourable for the $n$-th species, and belongs to this interval for ever.
We show that these two properties hold for the other species by induction on
the index set $\{n,n-1,\cdots,2\}$: this is our main result
(Proposition \ref{MainProp} given in Section \ref{secrecursion}). 
This proves that the only winner of the competition is the first species.

\section{Competitive exclusion principle for the chemostat}

The classical chemostat model for $N\in\mathbb{N}^*$ species competing
for a single resource is given by the system of differential equations

\begin{equation}
\label{sys_n_souches}
\left\{
\begin{array}{lll}
\dot s = &  D(S_{\rm in}-s)-\sum_{i=1}^N \mu_{i}(s)x_{i}\\[4mm]
\dot x_{i} =  & \mu_{i}(s)x_{i}-D x_{i} \quad (1\leq i\leq N)
\end{array}\right. ,
\end{equation}\
where the operating parameters $D>0$, $S_{\rm in}>0$ are the
removal rate and the incoming density of resource (or input substrate
concentration). The variable $s(t)$ denotes the density of resource
(or substrate concentration) at time $t$. For $1\leq i \leq N$, $x_i(t)$
represents the density of the $i$-th species and $\mu_i(\cdot)$ is
the specific growth rate function of species $i$. 
In this writing we have assumed, without any loss of generality, that
the yield coefficients of resource $s$ transformed in $x_{i}$ are all
identically equal to $1$.
In microbial ecology, the growth function $\mu_i$ often takes the form
of a Monod's function $\mu(s)= \mu_{\rm max} \frac{s}{k+s}$, but we
consider here more general ecosystems without particularizing the
expression of the growth function. We make the following assumption.

\begin{assumption}[Growth function]
\label{assumption1}
For each $1\leq i \leq N$, we assume that
\begin{enumerate}
\item $\mu_i\in C^1(\mathbb{R}^+)$.
\item $\mu_i(0)=0$.
\item $\mu_i(\cdot)$ is an increasing of function  of $s$. 
\end{enumerate}
\end{assumption}

One can straightforwardly check that the positive orthant of
$\Rset^{N+1}$ is invariant by the dynamics \eqref{sys_n_souches}.

As often considered in the literature, we associate to each growth
function the {\em break-even concentration} defined as follows.

\begin{definition}[Break-even concentration] \label{def_bec}
Under Assumption \ref{assumption1}, for a given number $D>0$, the break-even concentration $\lambda_i=\lambda_i(D)$ for the $i$-th species is defined as the unique solution of the equation $\mu_i(s)=D$, when it exists. When there is no solution to this equation, we set $\lambda_i =\infty$. 
\end{definition}

%
\begin{assumption}\label{assumption2} 
Species have distinct break-even concentrations, and without loss of
generality are enumerated by indices such that
\begin{equation}\label{orderBEC}
\lambda_1<\lambda_2< \cdots < \lambda_N.
\end{equation}
\end{assumption}

In Figure \ref{growthsFunctionsFigure}, we have represented several
growth functions and the removal rate (in dashed line).
\begin{figure}[ht!]
\begin{center}
\includegraphics[scale=0.7]{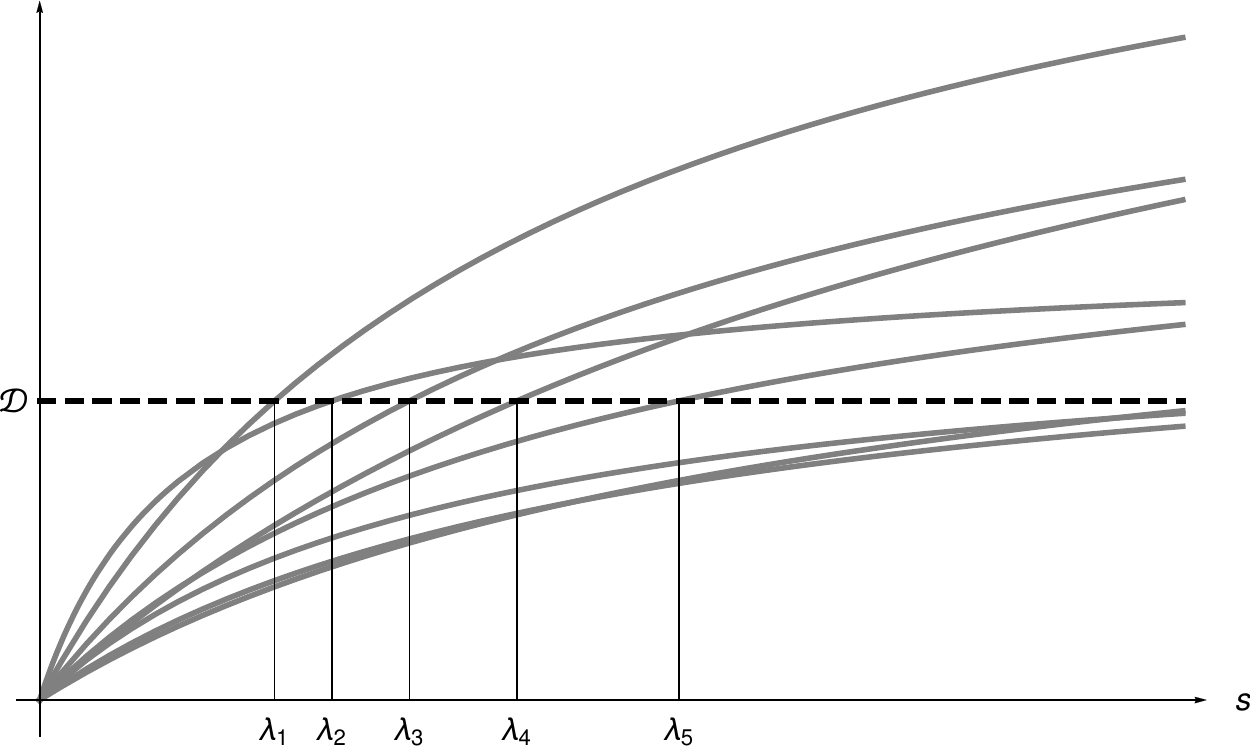}
\caption{Growth functions and their break-even concentrations.}
\label{growthsFunctionsFigure}
\end{center}
\end{figure}
 
We shall see further how the particular cases when some numbers
$\lambda_{i}$ are identical can be tackled, packing the
corresponding species (see Section \ref{sec_identical} below). 
We first do not consider these non generic situations for sake of
simplicity of the presentation.

\medskip

We recall the statement of the Competitive Exclusion Principle.

\begin{prop}[Competitive exclusion principle]\label{mainresult} 
Assume one has $\lambda_{1}<S_{\rm in}$ with Assumptions \ref{assumption1} and \ref{assumption2} fulfilled.
For any non-negative initial condition with $x_1(0)>0$, the solution of the system \eqref{sys_n_souches} converges to the equilibrium point $(\lambda_1,S_{\rm in}-\lambda_1,0,\ldots,0)$.
\end{prop}

\section{Proof of Proposition \ref{mainresult}}

The proof is based on a backward inductive argument: we show that the
proportion, with respect of the total biomass, of each species goes to
$0$ when $t$ tends to $+\infty$ excepted for the species with minimal break-even concentration.

\subsection{Change of coordinates}
We introduce the \emph{total biomass} $b$ and the \emph{proportion's vector} $p:=(p_i)_{1\leq i \leq N}$ where:

\begin{equation}
b:= \sum_{i=1}^N x_i \quad {\rm and} \quad p_i := \frac{x_i}{b}.
\end{equation}

Additionally, we define the function $\bar\mu(s,p):= \sum_{i=1}^N p_i \mu_i(s)$. In those new variables, one can easily check that the system \eqref{sys_n_souches} writes:

\begin{equation}
\label{sys_n_sbp}
\left\{
\begin{array}{lll}
\dot s & = &  D(S_{\rm in}-s)-\bar\mu(s,p)b\\
\dot b & = & \bar\mu(s,p)b-Db\\
\dot p_{i} & =  & p_{i}(\mu_{i}(s)-\bar\mu(s,p)) \qquad
(1\leq i\leq n).
\end{array}\right.
\end{equation}

\subsection{Non extinction of the total biomass}

We first give a necessary and sufficient condition for the persistence
of the total biomass, of interest in itself.

\begin{lem}\label{Lemme1}
Consider that Assumptions
\ref{assumption1} and \ref{assumption2} are fulfilled.
\begin{enumerate}
\item Whatever is the initial condition of \eqref{sys_n_souches} , the solution verifies \[\lim_{t\to+\infty} b(t)+s(t)=S_{in}.\]
\item Any species $j$ with $\lambda_{j}\geq S_{\emph{in}}$ satisfies
  $\lim_{t\to+\infty} x_{j}(t)=0$, whatever is the initial condition
  of \eqref{sys_n_souches}.
\item When $\lambda_{1}<S_{\rm in}$, for any non-negative initial condition
such that there exists $i \in \{1,\cdots,n\}$ with
$x_{i}(0)>0$ and $\lambda_{i}<S_{\rm in}$, the variable $b(t)$ is
bounded from below by a positive number for any $t>0$. 
\end{enumerate}
\end{lem}

\begin{proof}

Consider the \emph{total mass}  $m:=b+s$ of the system. Then
$m$ solves the linear differential equation $\dot{m}=D(S_{\rm in}-m)$, which
posses an unique equilibrium $m^\star= S_{\rm in}$, that is  moreover
globally asymptotically stable.

\medskip

Consider first species $j$ such that $\lambda_{j}\geq S_{\rm in}$ (if
it exists). Fix $\epsilon>0$. Then $\eta=D-\mu_{j}(S_{\rm
  in}-\epsilon/2)$ is a positive number.
As $m(\cdot)$ converges to $S_{\rm in}$,
there exits $T>0$ such that $s(t)+x_{j}(t)\leq m(t)<S_{\rm in}+\epsilon/2$
for any $t>T$. This implies that the dynamics of species $j$ satisfies
\begin{equation*}
\dot x_{j}(t)\leq \left(\mu_{j}(S_{\rm in}-x_{j}(t)+\epsilon/2)-D\right)x_{j}(t)=:\phi_{j}(x_{j}(t))
\end{equation*}
for any $t>T$. The function $\phi_{j}$ has the property
\begin{equation*}
x_{j}\geq \epsilon  \; \Rightarrow \; \phi_{j}(x_{j}) \leq 
\left(\mu_{j}(S_{\rm in}-\epsilon/2)-D\right)x_{j}\leq -\eta\epsilon <0.
\end{equation*}
Therefore, the variable $x_{j}(t)$ exits the domain $\{
x_{j}\geq \epsilon\}$ in finite time and stays outside for any future
time, i.e. there exists $T'>T$ such that $x_{j}(t)<\epsilon$ for any
$t >T'$. This statement is obtained for any arbitrary $\epsilon$, which proves
the convergence of the $x_{j}(\cdot)$ towards $0$.

\medskip 

Let $k$ be the maximal index such that $\lambda_{k}<S_{\rm in}$
(which exists by Assumption \ref{assumption2}) and
denote $b_{0}:=x_{1}+\cdots +x_{k}$.
For any $i\leq k$, notice that one has
\begin{equation*}
\mu_{i}(s)x_{i} \geq \min_{1\leq j \leq k} \mu_j(s) x_i
\end{equation*}
for any $s$ and $x_{i}$, and by simple addition, 
\begin{equation}\label{here}
\sum_{i=1}^k \mu_i(s) x_i  \geq  \min_{1\leq j \leq k} \mu_j(s) \Big(
\sum_{i=1}^k x_i \Big)= \min_{1\leq j \leq k} \mu_j(s) b_{0}.
\end{equation}
The dynamics of $b_{0}$ writes
\begin{equation*}
\dot{b}_{0} = \sum_{i=1}^k \dot{x}_i = \sum_{i=1}^k
\mu_i(s)x_i-D\sum_{i=1}^k x_{i} \\
\end{equation*}
and by inequality \eqref{here}, one has $\dot{b}_{0}(t) \geq
\Psi(t,b_{0}(t))$ for any $t>0$,
where $\Psi$ is defined as follows
\begin{equation} \label{defPsi}
\Psi(t,b_{0}):=\big(\min_{1\leq i \leq k} \mu_i(s(t))-D\big)\, b_{0}.
\end{equation}
As the inequality $\mu_{i}(S_{\rm in})>D$ is
fulfilled for any $i \leq k$, and the functions
$\mu_{i}$ are continuous, there exists
$\epsilon>0$ and $\eta>0$ such that
\begin{equation} \label{ineqsigma}
\min_{1\leq i \leq k} \mu_i(\sigma)-D>\epsilon \mbox{ for any } \sigma>S_{\rm in }-\eta.
\end{equation}
Since $m(t)$ converges to $S_{\rm in}$ and $x_{j}(t)$ converges to $0$ for
any $j>k$, there is a time $T^\star\geq 0$ such that 
\begin{equation*}
s(t)=m(t)-b_{0}(t)-\sum_{j\geq k}x_{j}(t)>S_{\rm in}-b_{0}(t)-\frac{\eta}{2} \mbox{ for
  any } t > T^\star.
\end{equation*}
Then, by inequality \eqref{ineqsigma}, the function $\Psi$ defined in \eqref{defPsi} fulfills the following property.
\[ 
\left\{ t>T^\star, b_{0} < \tfrac{\eta}{2} \right\} \Rightarrow \Psi(t,b_{0}) \geq
\big( \min_{1\leq i \leq k} \mu_i(S_{\rm in}-\eta)-D \big)b_{0} \geq \epsilon \, b_{0}
.\]
As one has $b_{0}(0)>0$ by hypothesis, $b_{0}(t)$ is strictly positive for any
$t>T^\star$ and it follows
that $b_{0}$ can not stay or enter into the interval $[0,\tfrac{\eta}{2}]$
for times larger than $T^\star $. Then the total biomass 
$b(t)\geq b_{0}(t)$ is bounded from below by $\tfrac{\eta}{2}$ in finite time.
\end{proof}

\subsection{Frame on the substrate's dynamics}

By Assumption \ref{assumption2}, growth functions are ordered on each
interval $[\lambda_{i},\lambda_{i+1}]$  (see Figure
\ref{growthsFunctionsFigure}) i.e. for all $i\in\{1\cdots
n-1\}$, one has
\[
\mu_{i}(s)>\mu_{j}(s) , \quad \forall j > i , \quad \forall s \in
[\lambda_{i},\lambda_{i+1}].
\]
By continuity of $\mu_{i}(\cdot)$, there are numbers $\nu>0$,  $s_{i}^-<\lambda_{i}$ and  $s_{i}^+>\lambda_{i+1}$ such that
\begin{equation}
\label{nu}
\mu_{i}(s)>\mu_{j}(s)+\nu ,  \quad \forall s \in [s_{i}^-,s_{i}^+],
\quad \forall j > i.
\end{equation}

By monotonicity of $\mu_{i}$, we have $\mu_{i}(s_{i}^-)<D$ and
$\mu_{i}(s_{i}^+)>D$. Moreover, one has
\begin{eqnarray*}
j<i  & \Rightarrow &
\mu_{j}(s_{i}^+)>\mu_{j}(\lambda_{i})>\mu_{j}(\lambda_{j})=D\\
s_{i}^+ > \lambda_{i+1} & \Rightarrow & \mu_{i+1}(s_{i}^+) > D
\end{eqnarray*}
Therefore, the numbers
\begin{equation}
\label{gamma_pm}
\left\{\begin{array}{lll}
\gamma^- & := & D-\mu_{1}(s_{1}^-)\\
\gamma^+ & := & \min_{2\leq i\leq  n}\min_{j<i}\mu_{j}(s_{i-1}^+)-D
\end{array}\right.
\end{equation}
are positive. We define also the following numbers.
\begin{equation}
\label{D_pm}
D^-:=D-\frac{\gamma^-}{2} , \quad D^+:=D+\frac{\gamma^+}{2}  \ .
\end{equation}

\begin{lem}\label{PropSubstratEncadrement}
\label{Lemme2}
Assume one has $\lambda_{1}<S_{\rm in}$ with Assumptions
\ref{assumption1} and \ref{assumption2}. For any non-negative
initial condition with $x_{1}(0)>0$, there exists $T>0$ such
that $\dot s(t) \in [\Phi^-(t,s(t)),\Phi^+(t,s(t))]$ for any $t>T$, where
\begin{equation}
\label{defPhi}
\left\{\begin{array}{rll}
\Phi^-(t,s) & := & [D^{-}-\bar\mu(s,p(t))]\,b(t)\\
\Phi^+(t,s) & := & [D^{+}-\bar\mu(s,p(t))]\,b(t).
\end{array}\right.
\end{equation}
\end{lem}

\begin{proof}
From Lemma \ref{Lemme1}, we know that $b(\cdot)$ is bounded from below
by a positive number. Consider the function
\[
z:=\frac{S_{in}-s}{b}=1+\frac{S_{in}-m}{b} .
\]
Since $m(\cdot)$ converges to $S_{in}$ and $b(\cdot)$ is bounded form
below, $z(t)$ converges to $1$ when $t$ tends to $+\infty$. 
Furthermore, for $t>T$ with $T$ large enough, one has $Dz(t)
\in [D^-,D^+]$.
Remark that the substrate dynamics can be written as follows
\[
\dot s = \Phi(t,s) := b(t)\,[Dz(t)-\bar\mu(s,p(t))] \ ,
\]
and one then obtains the inequalities $\Phi^-(t,s)\leq \Phi(t,s)\leq \Phi^+(t,s)$
for any $t>T$.
\end{proof}

\subsection{Extinction of the species by induction}
\label{secrecursion}

\begin{prop}\label{MainProp}
Assume one has $\lambda_{1}<S_{\rm in}$ with Assumptions
\ref{assumption1} and \ref{assumption2}.
Consider a solution of system \eqref{sys_n_souches} with $x_{1}(0)>0$.
Then the property
\[
({\mathcal{P}}_{i}): \quad 
\left\{
\begin{array}{ll}
i. & \ds \mbox{There exists } T_{i}>0 \mbox{  s.t. } s(t) \in I_{i}:=[s_{1}^-,s_{i}^+] , \; \forall t > T_{i}\\
ii. & \ds \lim_{t\to+\infty} p_{j}(t)= 0 , \; \forall j>i .
\end{array}\right.
\]
holds for all $i\in \{1,\ldots, n-1 \}$.
\end{prop}

Figure \ref{figIi} illustrates the sets $I_{i}$ considered in the
backward induction of the proof.
\begin{figure}[ht!]
\begin{center}
\includegraphics[width=0.9\textwidth]{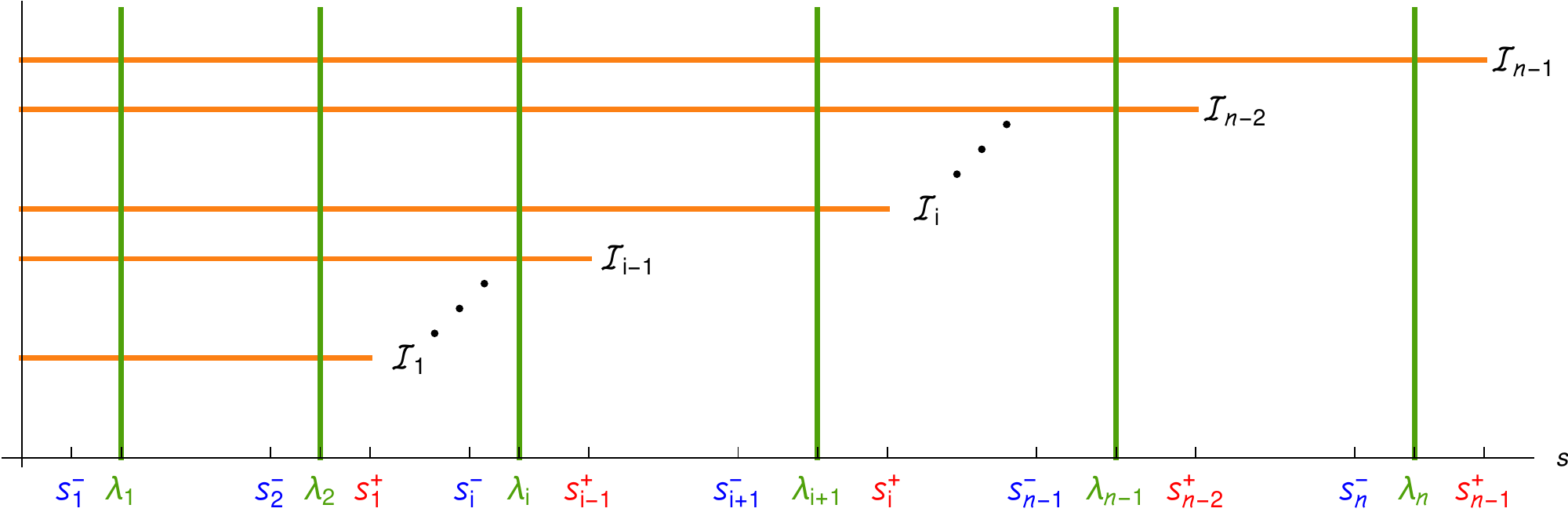}
\caption{Illustration of the intervals $I_{i}$ ($i=1\cdots n-1$)
\label{figIi}}
\end{center}
\end{figure}

\begin{proof}
As $x_{1}(0)>0$, the variable $x_{1}$ stays positive for any time and one can consider variables $r_{i}:=\tfrac{x_{i}}{x_{1}}$ for $i\in\{2,\cdots, n\}$,
whose dynamics is
\begin{equation}
\label{dyn_ri}
\dot r_{i}=[\mu_{i}(s(t))-\mu_{1}(s(t))]r_{i}.
\end{equation}
By monotonicity of functions $\mu_{i}(\cdot)$ and property \eqref{nu},
we can write
\begin{equation}
\label{propDmoins}
s<s_1^- \Rightarrow \max_{i}\mu_{i}(s) <
\max_{i}\mu_{i}(s_{1}^-)=\mu_{1}(s_{1}^-)
\leq D^--\frac{\gamma^-}{2} 
\end{equation}
where $\gamma^-$, $D^-$ are defined in \eqref{gamma_pm}, \eqref{D_pm}.
Similarly, one has
\begin{equation}
\label{propDplus}
s > s_{n-1}^+ \Rightarrow 
\min_{i} \mu_{i}(s) > \min_{i}\mu_{i}(s_{n-1}^+)
 \geq D^++\frac{\gamma^+}{2}
\end{equation}
where $\gamma^+$, $D^+$ are defined in \eqref{gamma_pm},
\eqref{D_pm}.

Consider a number $T>0$ given by Lemma \ref{Lemme2}. 
From Lemma \ref{Lemme1}, there exists a number $\eta>0$ such
that $b(t)\geq \tfrac{\eta}{2}$ for any $t>T$. 
It then follows from \eqref{propDmoins} and
(\ref{propDplus}) that the functions $\Phi^{\pm}$  defined in
\eqref{defPhi}
fulfill the following inequalities.
\begin{eqnarray*}
s<s_{1}^-, \;  t>T & \Rightarrow &
\Phi^-(t,s) \geq
b(t)\left(D^--\max_{i}\mu_{i}(s)\right)>\frac{\eta\,\gamma^-}{4}>0\\
s>s_{n-1}^+, \; t>T & \Rightarrow &
\Phi^+(t,s) \leq
b(t)\left(D^+-\min_{i}\mu_{i}(s)\right)<-\frac{\eta\,\gamma^+}{4}<0.
\end{eqnarray*}
Therefore the variable $s$ enters the interval $I_{n-1}$ in a finite time
$T_{n-1}>T$ and belongs to it for any future time.
Furthermore, the inequalities \eqref{nu} ensure to have 
$\mu_{n}(s)-\mu_{1}(s)<-\nu$ for any $s$ in the interval $I_{n-1}$, and then the
dynamics of $r_{n}$ satisfies
$\dot{r}_{n}\leq -\nu r_{n}$ for $t>T_{n-1}$. Thus $r_{n}$ converges
to $0$, and $x_{n}$ converges as well. Property
${\mathcal{P}}_{n-1}$ is then satisfied. 

\medskip

Assume that $\mathcal{P}_{i}$ is satisfied for an index $i \in
\{2,\cdots,n-1\}$ and let us show that $\mathcal{P}_{i-1}$ is
fulfilled.
Since the functions $p_{j}$ converge to $0$ for all $j>i$, there is $T'>T_{i}$ such that  
\[
\sum_{j>i} p_{j}(t) < \varepsilon:=\frac{\gamma^+/4}{D^+ + \gamma^+/2}, \quad
\forall t> T' .
\]
Then, for  $s>s_{i-1}^+$ and $t>T'$, the following inequality holds
\[
\begin{array}{lll}
\bar\mu(s,p(t)) & \geq & \sum_{j\leq i}\mu_{j}(s)p_{j}(t) \\
&\geq & \min_{j\leq i}
\mu_{j}(s_{i-1}^+)(1-\varepsilon)\\
&\geq& (D^+ +\tfrac{\gamma^+}{2})(1-\varepsilon)=D^++\frac{\gamma^+}{4}
\end{array}
\]
which provides the property of the function $\Phi^+$:
\[
s>s_{i-1}^+, \; t>T' \; \Rightarrow \; \Phi^+(t,s) \leq 
-\frac{\eta\,\gamma^+}{8}<0. 
\]
Thus, in a finite time $T_{i-1}>T'$, $s$ enters into the interval
$I_{i-1}$ and stays inside it for any future time. Furthermore, inequalities
\eqref{nu} lead to write $\dot{r}_{i}\leq-\nu r_{i}$ for
$t>T_{i-1}$, which shows that $x_{i}$ converges to zero.
Property $\mathcal{P}_{i-1}$ is then satisfied.
\end{proof}

\subsection{The case of identical break-even concentration}
\label{sec_identical}

We relax Assumption \ref{assumption2} allowing some $\lambda_i$  with $i>1$ to be identical and show that Proposition \ref{mainresult} is also satisfied.

If there exist $i$ and $\ell>1$ such that
$\lambda_1<\lambda_{i-\ell}(D)= \cdots = \lambda_i(D)<S_{\rm in}$, at
step $i$ in the induction of the proof, we replace species $i$ by the sum of species $i-\ell, \ldots, i$ and show property $\mathcal{P}_i$, that is  $\lim_{t\to + \infty} p_j(t)=0$ for all $j>i-\ell+1$. 

This can be done considering $\overline{r}_i:= \sum_{j=i-\ell}^i r_j$
instead of $r_i$, and remark that we have also $s_{i-1}^+ = \cdots =
s_{i-1-\ell}^+$. 
Then, one has
\[ 
\dot{\overline{r}_i} = \left[ \sum_{j=i-\ell}^{i}
  \alpha_j(t)\mu_j(s(t)) -\mu_1(s(t)) \right] r_i
\]  
where $\alpha_j(t) = \frac{r_j(t)}{\overline{r}_i(t)} >0$ with $\sum_{j=i-1}^{i} \alpha_j(t)=1$.  By 
 \eqref{nu}, we have
\[
\sum_{j=i-1}^i \alpha_j(t) \mu_j(s)< \mu_1(s)-\nu , \quad
\forall s\in [s_1^-,s_i^+]
\]
an then
 $\dot{\overline{r}}_i(t)<-\eta \overline{r}_i(t)$. Thus
 $\dot{\overline{r}_i} \leq \eta\, \overline{r}_i$ for $t>T$ which
 shows that $\overline{r}_i$ converges to zero as well as every $x_j$ with $i\leq j \leq i+\ell$. The rest of the proof is identical.


\subsection{Conclusion}
Now that we know that the variables $x_{j}$ converge to $0$ for any
$j>1$, we can easily show that $s$ tends to $\lambda_{1}$ as follows.

\medskip

By Lemma \ref{Lemme1} and Proposition \ref{MainProp}, $b(t)+s(t)$ tends to
$S_{\rm in}$ and the variables $b$ and $s$ are bounded from below by
positive numbers. Thus there exists numbers $\zeta>0$ and $T_{1}>0$
such that $s(t) \in [\zeta,S_{\rm in}-\zeta]$ for any $t>T_{1}$, and
accordingly to Proposition \ref{MainProp}, we can require to have $\lambda_{1} \in
(\zeta,S_{\rm in}-\zeta)$. We then consider the variable
\[
r(t)=\frac{\bar\mu(s(t),p(t))b(t)}{\mu_{1}(s(t))(S_{\rm in}-s(t))},
\quad t>T_{1}
\]
which tends to $1$ when $t$ tends to $+\infty$, and write the dynamics 
of $s$ as
\begin{equation}
\dot s = \Gamma(t,s) := (S_{\rm in}-s)(D-\mu_{1}(s)r(t))
\end{equation}
Take any $\epsilon>0$ sufficiently small to have
$[\lambda_{1}-\epsilon,\lambda_{1}+\epsilon] \subset [\zeta,S_{\rm in}-\zeta]$.
The function $\mu_{1}(\cdot)$ being assumed to be of class $C^1(\mathbb{R}^+)$, increasing and with
$\mu_{1}(\lambda_{1})=D$, there exists
$T_{2}>T_{1}$ and $\eta>0$ such that 
\begin{equation*}
\begin{array}{l}
t>T_{2}, \; s > \lambda_{1}+\epsilon \Rightarrow \mu_{1}(s)r(t)>D+\eta\\
t>T_{2}, \; s < \lambda_{1}-\epsilon \Rightarrow \mu_{1}(s)r(t)<D-\eta
\end{array}
\end{equation*}
Then, the function $\Gamma$ fulfills the following properties
\begin{equation*}
\begin{array}{l}
t>T_{2}, \; s \in [\lambda_{1}+\epsilon,S_{\rm in}-\zeta] \Rightarrow \Gamma(t,s)<-\eta\zeta<0\\
t>T_{2}, \; s \in [\zeta,\lambda_{1}-\epsilon] \Rightarrow
\Gamma(t,s)>\eta(S_{\rm in}-\zeta)>0
\end{array}
\end{equation*}
which allows to conclude that the variable $s$ converges to the
interval $[\lambda_{1}-\epsilon,\lambda_{1}+\epsilon]$, and this can
be obtained for any arbitrarily small $\epsilon>0$.

\section*{Acknowledgments}
The authors are thankful for having fruitful exchanges with
Tewfik Sari, Claude Lobry, Jérôme Harmand,
Kevin Cauvin and Romaric Condé about the
chemostat model. Mario Veruete is grateful for the support of the
National Council For Science and Technology of Mexico, and to his Ph.D.
supervisor, Dr. Matthieu Alfaro, for his continuous support.

\bibliographystyle{plain}
\bibliography{bibliography}

\end{document}